\documentclass{amsart}

\RequirePackage[OT1]{fontenc}
\RequirePackage{amsthm,amsmath,natbib}
\RequirePackage[colorlinks,citecolor=blue,urlcolor=blue]{hyperref}

\usepackage[english]{babel}

\usepackage{amssymb}
\usepackage{amsfonts,amsmath,amssymb,bm,latexsym,enumerate, mathrsfs,rotate, }

\def\SC { \mathscr{C}}

\def \FP{\mathfrak{q}}  

\def \EPk{\mathfrak{q}}

\def\CM { \mathcal{M}}

\def\CP { \mathcal{P}}

\def\CG {\mathcal{G}}
\def\CP {\mathcal{P}}

\def \CC{\mathcal{C}}

\def \FP{\mathfrak{q}}

\def\eps {\epsilon}

\def\J {\mathbf{1}}

\def\N {\mathbb{N}}

\def\E {\mathbb{E}}

\def\N {\mathbb{N}}
\def\P {\mathbb{P}}

\def \XX{\mathbb{X}}

\usepackage{color,comment}

\newcommand{\bn}{\boldsymbol{n}}
\newcommand{\bmm}{\boldsymbol{m}}

\newcommand{\blam}{\boldsymbol{ \lambda}}

\newcommand{\bc}{\boldsymbol{c}}

\newtheorem{definition}{Definition}
\newtheorem{thm}[definition]{Theorem}
\newtheorem{lemma}[definition]{Lemma}
\newtheorem{proposition}[definition]{Proposition}
\newtheorem{corollary}[definition]{Corollary}
\newtheorem{rem}{rem}

\title[Clustering structure for species sampling sequences]{Clustering structure for species sampling sequences with general base measure}

\author{Federico Bassetti$^*$}\thanks{$^*$ Department of Mathematics, Politecnico of Milano, Italy.  e-mail: \href{mailto:federico.bassetti@polimi.it }{federico.bassetti@polimi.it}}
\author{Lucia Ladelli$^{**}$}\thanks{$^{**}$ Department of Mathematics, Politecnico of Milano, Italy.  e-mail: \href{mailto:lucia.ladelli@polimi.it }{lucia.ladelli@polimi.it}}

\begin{document}

\maketitle

\begin{abstract}

We investigate the clustering structure of species sampling sequences $(\xi_n)_n$,
with general base measure. Such sequences are exchangeable with 
a species sampling random probability as 
directing measure. The clustering properties of these sequences are interesting 
for Bayesian nonparametrics applications, where  mixed base measures
are used, for example,  to accommodate sharp hypotheses in regression problems and provide sparsity. 
In this paper,  we prove a stochastic representation for  $(\xi_n)_n$ in terms 
of a latent exchangeable random partition. We provide explicit expression of the 
EPPF of the partition generated by $(\xi_n)_n$  in terms of the  EPPF of the latent partition.
We  investigate  the  asymptotic behaviour of 
the total number of blocks and of  the number of blocks  with fixed cardinality 
in the partition generated  by  $(\xi_n)_n$.

\end{abstract}

\section{Introduction}

{
Many 
important nonparametric priors, e.g. the Dirichlet  and the Pitman Yor process, 
can be seen as particular 
 {\it Species Sampling random probabilities}, that is random probabilities  of the form 
\begin{equation}\label{eq0}
   P= \sum_{j \geq 1} p_j \delta_{Z_j},
\end{equation}
   where
 $(Z_j)_{j \geq 1}$ are i.i.d. random variables taking values in a Polish space $\XX$ with common distribution $H$ 
 and $(p_j)_{j \geq 1}$ are random positive weights in $[0,1]$ independent from $(Z_j)_{j \geq 1}$.

With few exceptions, see e.g. \cite{RLP,Sangalli2006,Broderick2018},
the  {\it base measure} $H$ of such processes is usually assumed to be diffuse, 
since this simplifies the derivation of some analytical 
results. 
A sequence of random variables whose  directing measure is a  species sampling 
 random probability  (with diffuse base measure) is usually called {species sampling sequence}, 
 and the combinatorial structure of such sequences has  been deeply investigated, 
 see \cite{Pit06} and the references therein.

Recently 
mixed base measures appeared in Bayesian nonparametrics, since  
in various applications 
the available prior information leads naturally to the incorporation
of atoms into the base measure. For example,  in order 
to induce sparsity and facilitate variable selection, Dirichlet Processes with 
Spike-and-slab base  measures have  been used by many authors, see e.g.  
 \cite{Dunson08,Vannucci2009,suarez2016,Cui2012,Barcella2016}. 
Spike and slab  base measures have also been considered for a Pitman Yor process
in \cite{Canale2017}, where
computable expressions for the distribution of the 
random partitions  induced by such a process 
are derived and used for  predictive inference.

Motivated by the  recent interest in species sampling models
with spike and slab base measure, 
in  this paper we  discuss some relevant properties 
of random partitions induced by  species sampling sequences  
 with general base measure. 

 We prove a stochastic representation for  species sampling sequences with a general base measure in terms 
of a latent exchangeable random partition, Proposition \ref{prop_gsss} . We provide explicit expression of the 
{\it Exchangeable Partition Probability Function} (EPPF)
 of the 
  partition generated by such sequences in terms of the  EPPF of the latent partition, 
 Proposition \ref{prop_partition}.  The special case of 
 spike and slab base measure is further detailed in Proposition \ref{SpikeSlabEPPF}.
  Finally, 
  we  investigate  the  asymptotic behaviour of 
the total number of blocks and of  the number of blocks  with fixed cardinality 
of the partition induced by the sequence, Proposition \ref{PropAs1},
Theorems \ref{PropAs2} and \ref{PropAs3}.

Our   approach is different from the one used  in  
\cite {Sangalli2006} and \cite{Canale2017}, which is based on  
specific properties of 
nomalized  random measures. 
Using combinatorial arguments  developed in \cite{Pit06},
we are able to consider more general species sampling sequences and study their asymptotic properties.

\section{Species sampling sequences  with general base measure}
We start recalling  some basic concepts on  random partitions. More details and results are collected in the Appendix.
A partition $\pi_n$ of $[n]:=\{1,\dots,n\}$  is an unordered
collection $(\pi_{1,n},\dots,\pi_{k,n})$  of disjoint non-empty subsets (blocks) {of $\{1,\dots,n\}$}
 such that $\cup_{j=1}^k \pi_{j,n}=[n]$. 
 A partition $\pi_n=[\pi_{1,n},\pi_{2,n},\dots,\pi_{k,n}]$   has  $|\pi_n|:=k$ blocks (with $1 \leq |\pi_n|  \leq n$)  and  $|\pi_{c,n}|$, with $c=1,\dots,k$, is the number of  elements of the block 
$c$. 
We denote by 
 $\CP_n$  the collection of all  partitions of $[n]$ and, 
given a partition,  
we list its blocks in ascending order of their smallest element, i.e.  {\it in order of their appearance}. 
Given a  permutation $\rho$ of $[n]$ and  $\pi_n$ in $\CP_n$, denote by 
 $\rho(\pi_n)$ the partition with blocks $\{ \rho(j) : j  \in \pi_{i,n} \}$ for  $i=1,\dots,|\pi_n|$.
A  sequence of random partitions, $\Pi=(\Pi_n)_{n \geq 1}$, is called 
 {\it random partition of $\N$} if 
for each $n$ the random variable  $\Pi_n$ takes  values in $\CP_n$ 
and, for $m < n$, the restriction of $\Pi_n$  to $\CP_m$ is $\Pi_m$ ({\it consistency property}).
A random partition of $\N$  is said to be {\it exchangeable} 
if 
 $\Pi_n$ 
has the same distribution of $\rho(\Pi_n)$ 
  for every $n$ and every permutation $\rho$ of $[n]$. 

The law of any exchangeable random partition on $\N$ is characterized by its
 {\it Exchangeable Partition Probability Function} (EPPF), that is
there exists a unique symmetric function $\EPk$ on the integers\footnote{
 An EPPF can be seen as a family of 
 symmetric functions $\EPk_{k}^n(\cdot)$ defined on 
 $\CC_{n,k}=\{ (n_1,\dots,n_k) \in \N^k: \sum_{i=1}^k n_{i}=n\}$. 
 To lighten the notation we simply write $\EPk$ in place of $\EPk_{k}^n$.
 Alternatively, one can think that $\EPk$ is a function on 
 $\cup_{n \in \N} \cup_{k=1}^n \CC_{n,k}$.
 } such that, for any partition $\pi_n$ in $\CP_n$
\begin{equation}\label{EPPF1}
\P\{ \Pi_n=\pi_n\}=\EPk\left(|\pi_{1,n}|,\dots,|\pi_{k,n}|\right)
\end{equation}
 where $k$ is the number of blocks in $\pi_n$. See \cite{Pit06}.

%
%
 
 Kingman's correspondence theorem   (see Proposition \ref{kingmanth} in  Appendix)  sets up a one-to-one correspondence between 
 the law of an exchangeable  random partition on $\N$ (i.e. its 
EPPF) and 
 the law of random ranked weights $(p_j^{\downarrow})_{j \geq 1}$ satisfying 
  $1 \geq p_1^{\downarrow} \geq p_2^{\downarrow} \geq \dots \geq 0$
 and $\sum_jp_j^{\downarrow} \leq 1$ (with probability one).     

%
Given the {\it Species Sampling random probability} \eqref{eq0},
if $(p_j^{\downarrow})_{j \geq 1}$ is the ranked sequence obtained from  $(p_j)_{j \geq 1}$, one can 
always 
write  
\[
P
= \sum_{j \geq 1} p_j^{\downarrow} \delta_{Z_j'}
\]
where $(Z_j')_{j \geq 1}$ is a suitable random reordering of the original sequence $(Z_j)_{j \geq 1}$. 
It is plain to check that  $(Z_j')_{j \geq 1}$  are i.i.d. random variables with law $H$ independent from 
$(p_j^{\downarrow})_{j \geq 1}$. Hence, $H$ and   the EPPF $\FP$ 
associated via Kingman's correspondence to  $(p_j^{\downarrow})_{j \geq 1}$ completely characterize  the law of 
 $P$, from now on denoted by 
$SSrp(\FP,H)$.

Note that  in \eqref{eq0} we implicitly assume 
$\sum_j p_j=1$ almost surely, and hence we  are not considering 
the most general form of 
species sampling models, see \cite{Pit06}.

We lastly say that  a sequence $\xi=(\xi_n)_{n \geq 1}$ is a 
{\it generalized species sampling sequence}, 
 $gSSS(\FP,H)$, if the variables 
$\xi_n$ are conditionally i.i.d. given  $P$ from some  $P \sim SSrp(\FP,H)$ or, equivalently, if 
the directing measure of  $(\xi_n)_{n \geq 1}$ is a $SSrp(\FP,H)$.
If  $(\xi_n)_{n \geq 1}$ is  a $gSSS(\FP,H)$ with $H$ diffuse,  then 
it is a  
Species Sampling Sequences in the sense of Definition 12  in  \cite{Pitman96} and 
the random partition $\Pi(\xi)$
\footnote{If $X=(X_j)_{j\geq 1}$ is a sequence of random variables,
$\Pi(X)$ denotes  the random  partition 
obtained by the equivalence classes under the random equivalence relation $i(\omega) \sim j(\omega) $ if and only if $X_i(\omega) = X_j(\omega)$.
}
 (induced by $(\xi_n)_{n \geq 1}$)
has   EPPF $\FP$, see  Proposition 13 in \cite{Pitman96}. 
If   $(\xi_n)_{n \geq 1}$ is a $gSSS(\FP,H)$ but  $H$ is not diffuse,
 the relationship between the random partition induced by the sequence $(\xi_n)_{n \geq 1}$ and the EPPF $\FP$ is not as simple as in the diffuse case. 
In order to understand this relation it is usefull to introduce, for a random partition $\Pi$, the random index $\SC_n(\Pi)$ denoting the block containing $n$, that is
\[
\SC_n(\Pi)=c  \text{ if $n \in \Pi_{c,n} $}
\]  
or equivalently  if $n \in \Pi_{c,j} $ for some (and hence all)  $j \geq n$.

The next proposition, which is a refinement of { Proposition 1 in \cite{BaCaRo}}, shows that 
even in the non diffuse case, a $gSSS(\FP,H)$ is strictly related to a random partition $\Pi$ with 
EPPF $\FP$. When $H$ is diffuse $\Pi$ is  the partition induced by $(\xi_n)_n$, 
while if  $H$ has atoms, $\Pi$ is  a latent partition strictly finer than 
the partition induced by $(\xi_n)_n$.

\begin{proposition}\label{prop_gsss}
For a sequence  $(\xi_n)_{n \geq 1}$, 
the following are equivalent:
\begin{itemize}
\item[(i)] $(\xi_n)_{n \geq 1}$ is a  $gSSS(\FP,H)$;
\item[(ii)]  for every $n \geq 1$,
\begin{equation*}\label{allocationxi}
  \xi_n=Z_{I_n},  
\end{equation*} 
where 
 $(Z_j)_{j \geq 1}$ are i.i.d. random variables with common distribution $H$, 
  $(p_j )_{j \geq 1}$ is a  sequence of random weights independent from $(Z_j)_{j \geq 1}$, 
 $(I_n)_{n \geq 1}$ are conditionally independent given $(p_j )_{j \geq 1}$ with 
 $P\{I_n=k|(p_j )_{j \geq 1},(Z_j)_{j \geq 1} \}=p_k$. Moreover, 
 the EPPF 
associated via Kingman's correspondence to  $(p_j^{\downarrow})_{j \geq 1}$ is $\FP$. 
\item[(iii)] for every $n \geq 1$
\begin{equation*}\label{allocationxi}
  \xi_n=Z'_{\SC_n(\Pi)},
\end{equation*}
where  $\Pi$ is a random 
partition  with EPPF  $\FP$, $(Z'_j)_{j \geq 1}$ are i.i.d. random variables with common distribution $H$, 
 $\Pi$ and  $(Z'_j)_{j \geq 1}$ are stochastically independent.
\end{itemize}
\end{proposition} 
\begin{proof}
Let $(\xi_n)_{n \geq 1}$ be an exchangeable sequence of conditionally i.i.d. random variables, 
given  $P=\sum_{j \geq 1} p_j \delta_{Z_j}$ with law $P$.
Set $p=(p_j)_{j \geq 1}$ and $Z=(Z_j)_{j \geq 1}$. 
On a suitable enlarged probability space one can define a sequence $(I_n)_{n\geq 1}$ of integer random variables 
such that the $I_n$s are conditionally independent given $[(\xi_n)_{n \geq 1},  p,Z]$
and, up to a set of probability zero, 
\[
\P \{  I_n=j  | \xi_n,  p,Z \}= \frac{p_j}{\sum_{\{ i  : Z_i=\xi_n\}}  p_i}   \J\{ Z_j=\xi_n\}. 
\]
Now 
\[
\begin{split}
\P \{  I_n=j , \xi_n \in A | p,Z  \} & = 
\int_A    \P \{  I_n=j  | \xi_n,  p,Z \}  \P\{ \xi_n \in dx | p,Z \} \\
& 
=\int_A \sum_m   \delta_{Z_m}(dx)  p_m   \frac{p_j}{\sum_{\{ i  : Z_i=Z_m\}}  p_i}   \J\{ Z_j=Z_m\}      \\
& 
= p_j    \frac{\sum_{m:Z_m=Z_j}  p_m}{\sum_{i:  Z_i=Z_j }  \tilde  p_i}  \J \{Z_j \in A\} 
=p_j   \delta_{Z_j}(A). 
\end{split}
\]
From this, it is easy to deduce 
\[
 P\{ I_n=j_n : n=1,\dots,N |   p, Z\}=P\{ I_n=j_n : n=1,\dots,N |   p \}=\prod_{n=1}^N p_{j_n}  \quad a.s.
\]
that also implies that  $(I_n)_{n \geq 1}$ and $(Z_n)_{n \geq 1}$ are stochastically independent, given $p$. 
Hence (i) yields (ii) since  
\[
 \xi_n=Z_{I_n} \quad \text{a.s.}
\]
Let us  prove that (ii) yields (iii). 
By (A1) in the Appendix, if $\Pi=\Pi(I)$ is the partition induced by $(I_n)_{n \geq 1}$, then 
$\Pi$ has EPPF $\FP$. 
Denote by $I^*_1=I_1,I^*_2,\dots,I_K^*$ (with $K\leq +\infty$) the distinct values of $(I_n)_{n \geq 1}$ in order of appearance, 
and set 
\[
Z'_n=Z_{I^*_n} \quad n=1,\dots,K.
\]
If $K<+\infty$,  define  $(Z'_{K+1},Z'_{K+2},\dots)$ as the remaining $Z_n$s in increasing order. 
Using the independence of $(Z_n)_{n \geq 1}$ and $(I_n)_{n \geq 1}$  and the fact that the $Z_n$s are identically distributed, it follows 
that $(Z_n')_n$ is a sequence of i.i.d. random variables with common distribution $H$ and that 
$(Z_n')_n$ and $(I_n)_{n \geq 1}$ are stochastically independent. 
To conclude note that, with probability one, $\displaystyle I^*_{\SC_n(\Pi)}=I_n$ and hence
\[
 \xi_n=Z_{I_n}=Z_{I^*_{\SC_n(\Pi)}}=Z'_{\SC_n(\Pi)}.
\]
Conversely let us show that (iii) yields (ii). Let $(p_j^{\downarrow})_{j\geq 1}$  be the weights obtained from $\Pi$ by 
\eqref{kingmanTh}
 in Appendix. According to  (A2) in Appendix, it is possible to define  integer valued random 
variables $I_1,I_2,\dots$  
conditionally i.i.d., given $(p_j^{\downarrow})_{j\geq 1}$, with conditional distribution 
$\P\{ I_n=j| p^{\downarrow} \} = p^{\downarrow}_j$
such that $\Pi=\Pi(I)$ a.s..  Hence $\displaystyle \SC_n(\Pi)=\SC_n(\Pi(I))$ and $\displaystyle I^*_{\SC_n(\Pi)}=I_n$, as above. Setting
\[
Z_m
:=\left \{ 
 \begin{array}{ll}
Z'_k & \text{if $I^*_k=m$} \\
Z''_m & \text{if $I^*_k\neq m \ \forall \,\,k$}, \\
 \end{array}
 \right . 
\]
with $ Z''_m, m=1,2,\dots,$ i.i.d., independent from everything else and $Z''_m\sim H$. Then the $Z_m$s satisfy all the required properties and, 
in particular,
\[
Z_{I_n}=Z_{I^*_{\SC_n(\Pi)}}=Z'_{\SC_n(\Pi)}.
\]

To conclude we show that (ii) yields (i). 
Set  $P=\sum_{j \geq 1} p_j \delta_{Z_j}$ and recall that $\xi_n=Z_{I_n}$ by assumption. 
Given the Borel sets, $A_1,\dots,A_n$, and the integer numbers $i_1,\dots,i_n$, then we have 
\[
\P\left\{\xi_1 \in A_1, \dots, \xi_n \in A_n, I_1=i_1,\dots,I_n =i_n \middle \lvert  P,  (p_n)_{n \geq 1}, \left(Z_n\right)_n \right\}=
\prod_{j=1}^n \delta_{Z_{i_j}} (A_j) p_{i_j},
\]
and by marginalising,  
\[
\P\left\{\xi_1 \in A_1, \dots, \xi_n \in A_n \middle \lvert   P,  (p_n)_{n \geq 1}, \left(Z_n\right)_n  \right\}=
\sum_{i_1 \geq 1, \dots, i_n \geq 1} \prod_{j=1}^n \delta_{Z_{i_j}} (A_j) p_{i_j}
=\prod_{j=1}^n P (A_j).
\]
Hence, 
\[
\P\left\{\xi_1 \in A_1, \dots, \xi_n \in A_n |    P \right\}=\prod_{j=1}^n P (A_j)
\]
almost surely. Since $\XX$ is Polish, we proved that, given $P$, $(\xi_n)_{n \geq 1}=(Z_{I_n})_n$
 are  i.i.d.  with common distribution $P$, i.e. $(\xi_n)_{n \geq 1}$ is 
a $gSSS(\EPk,H)$ for $\EPk$ the EPPF corresponding to $(p_n)_{n \geq 1}$. 
\end{proof}

A simple consequence of the previous proposition is the next 

\begin{corollary} Let  $(\xi_n)_{n \geq 1}$ be a  $gSSS(\FP,H)$.
For every $A_1,\dots,A_n$ Borel sets in $\XX$, 
\[
\P\left\{ \xi_1 \in A_1, \cdots, \xi_n \in A_n\right\}= \sum_{\pi_n \in \CP_n}
\EPk(|\pi_{1,n}|,\dots,|\pi_{k,n}|)
 \prod_{c=1}^{|\pi_n|} H( \cap_{j \in \pi_{c,n}} A_j ).
\] 
\end{corollary}

\begin{rem}[Chinese Restaurant]\label{RemChinese}
{\rm Proposition \ref{prop_gsss}} can be restated in terms of the well-known  Chinese Restaurant  methaphor. 
In this metaphor, the  observations $\xi_n$ (identified by the indices $n=1,2,\dots$) are attributed to ``customers'' of a "restaurant".
First ``customers'' are clustered 
according to ``tables'', which are then clustered in an higher hierarchy by means of ``dishes''.
The  first  step of the clustering process (the sitting plan) is driven  by the random partitions $\Pi$,
with EPPF $\FP$,  that is $\SC_n(\Pi)$ with $n=1,2,\dots$.  At the second level, the dish $Z_i$ of table $i$ is sampled 
from $H$, independently for $i=1,2,\dots$.
\end{rem}

\section{Partition induced by  Species Sampling Sequences with general base measure}

Let  
$\tilde \Pi$ be the  random  partition induced by a   $gSSS(\FP,H)$.
From Proposition \ref{prop_gsss} and Remark \ref{RemChinese}, it is clear that if $H$ has atoms, different "tables" can merge in the final clustering configuration
described by  $\tilde \Pi$. In other words, two observations (customers) can share the same value (dish)
because they sit at the same table or  because they sit in different tables but they both sample 
the same dish from $H$. 
This simple observation leads to  write the EPPF of 
the random partition   $\tilde \Pi$ using the EPPF of   
$\Pi$ and the probability of ties in a vector of i.i.d. random variables  drawn from $H$.

\subsection{Explicit expression of the EPPF}
To go further,  we need some more notation. 
Given a vector ${\bn}=(n_1,\dots,n_k)$ of integer numbers such that $n=\sum_{i=1}^k n_i$, set
  \[
 \CM({\bn})=\Big \{  \bmm=( m_{1},\dots,m_{k})  \in \N^k:  1 \leq m_{i} \leq n_{i } \Big \}
\]
and, for $\bmm$ in  $\CM({\bn})$, define  $|\bmm|=\sum_{i=1}^k m_i$ and 
\[
 \Lambda(\bmm)
 :=\left \{ 
\begin{array}{lll}
&  
\text{$ \blam=[\blam_{1},\dots,\blam_{k}]$ where $\blam_{i}=(\lambda_{i1},\dots,\lambda_{in_{i}}) \in \N^{n_{i}} $}:  \\
%
& \quad  \sum_{j=1}^{n_{i}} j \lambda_{ij}=n_{i},  \sum_{j=1}^{n_{i}}  \lambda_{ij}=  m_{i} \,\, \text{ for $i=1,\dots,k$ } \\
\end{array}
\right\}.
\]
For $\blam$ in  $\Lambda({\bmm})$, define 
\[
\bc(\blam):=\prod_{i=1}^k  \frac{n_{i }!}{\prod_{j=1}^{n_{i }} \lambda_{ij}! (j!)^{\lambda_{ij}}    }
\]
and, given the EPPF $\FP$, set
\[
\tilde \FP (\blam):=  \FP (n_{11},\dots,n_{1m_{1}},n_{21}, \dots,n_{k m_{k}} ),
 \]
where $(n_{11},\dots,n_{1m_{1}},\dots,n_{k m_{k}})$ is any sequence of  integer numbers such that 
 $ \sum_{c=1}^{m_{i}} n_{i c } = n_{i}$ for every $i$ and 
$\#\{ c: n_{ic} =j \}= \lambda_{ij} $ for every $i$ and $j$. 
Note that 
since the value of $\FP(n_{11},\dots,n_{1m_{1}},n_{21}, \dots,n_{k m_{k}} )$
depends only on the statistics $\blam$, $\tilde \FP (\blam)$ is well defined. See e.g. \cite{Pit06}.

Let us consider an i.i.d. sample of length $|\bmm|$ from $H$ and
denote by 
$H^{\#}(\bmm)$ the probability of getting
exactly $k$ ordered 
blocks  with cardinality $m_1,\dots,m_k$,  such that 
observations in each block are equal and observations in distinct blocks are different.  
In order to write  $H^\#(\bmm)$  explicitly, we  decompose $H$ as 
 \begin{equation}\label{acca}
 H(dx)= \sum_{i=1}^{+\infty} \bar a_i \delta_{\bar x_i}(dx)+ (1-a) H^{c}(dx)
\end{equation}
where 
 $\XX_0:=\{\bar  x_1,\bar  x_2,\dots\}$  is the  collection of points  with positive $H$ probability,
 $\bar a_i=H(\bar x_i)$, 
 $a=H(\XX_0) \in [0,1]$ and $H^c(\cdot)=H(\cdot \cap \XX_0^c)/H(\XX_0^c)$ 
is a diffuse probability measure on $\XX$. 

Given $\bmm=(m_1,\dots,m_k)$ in  $\CM({\bn})$  let 
$\bmm^*$ the vector containing all the elements  $m_i>1$ and let $r$ be its length, 
with possibly $r=0$ if $\bmm=(1,1,\dots,1)$, and define for $\ell\geq 0$
\[
A_{\bmm,\ell}=
\sum_{j_1\not = \dots \not =j_{r+\ell} } \bar a_{j_1}^{m_{1}^*} \dots \bar a_{j_r}^{m_{r}^*}   \bar a_{j_{r+1}} \dots  \bar a_{j_{r+\ell}} 
\]
with the convention that $A_{\bmm,0}=1$ when $r=0$. 
A simple combinatorial argument shows that 
\[
H^{\#}(\bmm) = \sum_{\ell=0}^{k-r} (1-a)^{k-\ell-r}  { k-r \choose \ell } A_{\bmm,\ell}.
\]

\begin{proposition}\label{prop_partition}
Let  $(\xi_n)_{n \geq 1}$ be a $gSSS(\FP,H)$.
Denote by $\tilde \Pi$ the  random  partition induced by $(\xi_n)_{n \geq 1}$.
 If $\pi_n=[\pi_{1,n}\dots,\pi_{k,n}]$ is a partition of $[n]$ with
$|\pi_{i,n}|= n_i$ ($i=1,\dots,k$) 
and ${\bn}=(n_1,\dots,n_k)$, then 
 \[
\P\{ \tilde \Pi_n=\pi_n \}= \sum_{ \bmm \in \CM({\bn}) } H^{\#}(\bmm)
\sum_{  \blam  \in \Lambda(\bmm) } 
\bc(\blam)  \tilde \FP (\blam).
\]
\end{proposition} 

\begin{proof}
In order to describe the partitions $\pi^*_n$ which can give rise to 
$\pi_n$, when some blocks  are merged, we define 
 $\CP_{\pi_n}(\blam )$ as the set of all the partitions 
in $\CP_n$ with $m_1+\dots + m_k=|\bmm|$  blocks such that
\begin{itemize}
\item there are $k$ subset of blocks containing  $m_1,\dots,m_k$ blocks respectively;
\item the union of the blocks in the $i$-th subset 
coincides with the $i$-th block of $\pi_n$ for $i=1,\dots,k$;
\item  the blocks in the $i$-th subset   are formed by 
$\lambda_{ij}$ blocks of $j$ elements for $j=1,\dots,n_i$.  
\end{itemize}
Moreover, if $\pi^*_n=[\pi_{1,n}^*,\dots,\pi_{|\bmm|,n}^*]$ is in $\CP_{\pi_n}(\blam )$
set $M(j)=i$ if $\pi_{j,n}^*$ is in the $i$-th subset of blocks.
Finally, for $Z_1',Z_2',\dots$ as in (iii) of Proposition \ref{prop_gsss},
write $ \{\pi^*_n \hookrightarrow \pi_n\}$ to denote the event 
\[
\{ Z_{j_1}'=Z'_{j_2}\, \text{if 
$M(j_1)=M(j_2)$ and  $Z_{j_1}'\not=Z'_{j_2}$ if $M(j_1)\not=M(j_2)$,  for every $1\leq j_1\leq j_2 \leq |\bmm|$}\}.
\]
Using (iii) in Proposition \ref{prop_gsss}
we may assume that  $\xi_n:=Z'_{\SC_n(\Pi)}$, obtaining 
\[
\{ \tilde \Pi_n=\pi_n\}= \cup_{ \bmm \in \CM({\bn}) } \cup_{  \blam  \in \Lambda(\bmm) } 
 \cup_{  \pi^*_n \in \CP_{\pi_n}(\blam ) } 
\{ \Pi_n= \pi^*_n,  \pi^*_n \hookrightarrow \pi_n\}.
\]
Hence, by independence,
\[
\P\{ \tilde \Pi_n=\pi_n\}= \sum_{ \bmm \in \CM({\bn}) } 
\sum_{  \blam  \in \Lambda(\bmm) } \sum_{  \pi^*_n \in \CP_{\pi_n}(\blam ) }  \P\{\Pi_n= \pi^*_n  \} H^{\#}(\bmm).
\]
Now $\P\{\Pi_n= \pi^*_n  \}=  \tilde \FP (\blam)$ for every $ \pi^*_n \in \CP_{\pi_n}(\blam )$. To conclude it suffices to observe
that the cardinality of  $\CP_{\pi_n}(\blam )$ is 
\[
\prod_{i=1}^k  \frac{n_{i }!}{\prod_{j=1}^{n_{i }} \lambda_{ij}! (j!)^{\lambda_{ij}}    }.
\]
See e.g. formula (11) in \cite{Pit95}.
\end{proof}


\begin{rem}
If $H$ is diffuse,  then $ H^{\#}(\bmm)=0$ for every $\bmm \not =(1,1,\dots,1)$. Hence
the above formula reduces to the familiar
\[
\P\{\tilde \Pi_n=\pi_n\}=\FP(|\pi_{n,1}|, \dots,|\pi_{n,k}|)=\P\{\Pi_n=\pi_n\}.
\]
\end{rem}

An important class of exchangeable random partitions is that of  Gibbs-type partitions, introduced in \cite{GP2006} and  characterized by the EPPF
\begin{equation}\label{gibssEPPF}
\EPk(n_1,\dots,n_k):=V_{n,k} \prod_{j=1}^k (1-\sigma)_{n_j-1},
\end{equation}
where $(x)_n=x(x+1)\dots (x+n-1)$ is the rising factorial (or Pochhammer polynomial), $\sigma<1$ and  $V_{n,k}$ are positive real numbers such that  $V_{1,1}=1$ and
\begin{equation*}
  (n-\sigma k) V_{n+1,k} + V_{n+1,k+1}=V_{n,k},\quad n\geq 1,\,\,1 \leq k \leq n. \label{Rec_V}
\end{equation*}

A noteworthy example
of Gibbs-type  EPPF is the so-called 
Pitman-Yor two-parameters  family.  It is defined by  
\begin{equation}\label{PYeppf}
\EPk(n_1,\dots,n_k):=\frac{\prod_{i=1}^{k-1} (\theta+i\sigma)}{(\theta+1)_{n-1}} \prod_{c=1}^k (1-\sigma)_{n_c-1},
\end{equation}
where $0 \leq \sigma < 1$ and $\theta >-\sigma$; or $\sigma<0$ and $\theta=|\sigma|m$ for some integer $m$, see
\cite{Pit95,Pit97}.

In order to state the next result, we recall that 
\begin{equation}\label{stirling}
\sum_{ \substack{(\lambda_{1},\dots,\lambda_{n}) \\  \sum_{j=1}^{n} j \lambda_{j}=n,  \sum_{j=1}^{n}  \lambda_{j}=k}}
  \prod_{j=1}^{n} [(1-\sigma)_{j-1}]^{\lambda_{j}} \frac{n!}{\prod_{j=1}^{n} \lambda_{i}! (j!)^{\lambda_{j}}    }
=S_\sigma(n,k)
\end{equation}
where $S_\sigma(n,k)$ is the generalized Stirling number of the first kind, see (3.12) in  \cite{Pit06}.
In the same book various equivalent definitions of generalized Stirling numbers  are presented.  

\begin{corollary}
Let   $\tilde \Pi$ as in {\rm Proposition \ref{prop_partition}} with $\EPk$ of Gibbs-type defined  in   \eqref{gibssEPPF}.
 If $\pi_n=[\pi_{1,n}\dots,\pi_{k,n}]$ is a partition of $[n]$ with
$|\pi_{i,n}|= n_i$ ($i=1,\dots,k$) 
and ${\bn}=(n_1,\dots,n_k)$, then 
 \[
\P\{ \tilde \Pi_n=\pi_n \}    =  \sum_{ \bmm \in \CM({\bn}) } H^{\#}(\bmm)
V_{n,|\bmm|}   \prod_{i=1}^k   S_\sigma(n_i,m_i).
\]
\end{corollary}

\begin{proof}
Combining Proposition \ref{prop_partition}
with \eqref{gibssEPPF} one gets
 \[
 \begin{split}
\P\{ \tilde \Pi_n=\pi_n \}  & = \sum_{ \bmm \in \CM({\bn}) } H^{\#}(\bmm)
V_{n,|\bmm|} 
\sum_{  \blam  \in \Lambda(\bmm) } 
  \prod_{i=1}^k    \prod_{j=1}^{n_i} [(1-\sigma)_{j-1}]^{\lambda_{i,j}} \frac{n_{i }!}{\prod_{j=1}^{n_{i }} \lambda_{ij}! (j!)^{\lambda_{ij}}    }\\
  & 
 =  \sum_{ \bmm \in \CM({\bn}) } H^{\#}(\bmm)
V_{n,|\bmm|}   \\ & \qquad \quad  \times \prod_{i=1}^k   
\sum_{ \substack{(\lambda_{i1},\dots,\lambda_{in_{i}}) \\  \sum_{j=1}^{n_{i}} j \lambda_{ij}=n_{i},  \sum_{j=1}^{n_{i}}  \lambda_{ij}=  m_{i}}}
  \prod_{j=1}^{n_i} [(1-\sigma)_{j-1}]^{\lambda_{i,j}} \frac{n_{i }!}{\prod_{j=1}^{n_{i }} \lambda_{ij}! (j!)^{\lambda_{ij}}    }
 \\
  & 
 =  \sum_{ \bmm \in \CM({\bn}) } H^{\#}(\bmm)
V_{n,|\bmm|}   \prod_{i=1}^k   S_\sigma(n_i,m_i).
\\
  \end{split}
\]
\end{proof}

\subsection{Species sampling sequences  with Spike and Slab base measure
}\label{S:spikeslab}
A spike-and-slab  measure is defined as 
 \begin{equation}\label{SLH0}
H(dx)=a \delta_{x_0}(dx)+ (1-a) H^c(dx)
\end{equation}
where $a \in (0,1)$, $x_0$ is a point of $\XX$ and $H^c$ is a diffuse measure on $\XX$. 
This type of  measures has been used as base measure  by 
 \cite{Dunson08,Vannucci2009,suarez2016,Cui2012,Barcella2016}
in the Dirichlet 
Process and by \cite{Canale2017} in the Pitman-Yor 
process. 

Here we deduce 
by Proposition \ref{prop_partition} the explicit form 
of  the EPPF of the random partition 
induced by a sequence sampled from a species sampling random probability with such a base measure.

\begin{proposition}\label{SpikeSlabEPPF} Let $H$ be as in \eqref{SLH0}, 
 $\tilde \Pi$ be the  random  partition induced by 
a $gSSS(\FP,H)$ and $\Pi$ be an exchangeable 
random partition with EPPF $\FP$. 
If 
$\pi_n=[\pi_{1,n}\dots,\pi_{k,n}]$ is a partition of $[n]$ with
$|\pi_{i,n}|= n_i$ ($i=1,\dots,k$), then 
\begin{equation}\label{tilde_p_dens}
\begin{split}
\P\{ &  \tilde \Pi_n=\pi_n\}  = (1-a)^{k}   \FP(n_1,\dots,n_k) \\ 
& +(1-a)^{k-1}
 \sum_{i=1}^k \FP(n_1,\dots, n_{i-1},n_{i+1},\dots ,n_k )
 \sum_{r=1}^{n_i}  a^r
 q_n(r|n_1,\dots, n_{i-1},n_{i+1},\dots ,n_k )
\end{split}
\end{equation}
  where, conditionally on the fact that $\Pi_{n-n_i}$ has $k-1$ blocks with sizes $n_1,\dots, n_{i-1},n_{i+1},\dots,n_k$, 
   the probability that $\Pi_n$ 
has $k-1+r$ blocks is denoted by $q_n(r|n_1,\dots, n_{i-1},n_{i+1},\dots,n_k )$. 
If in addition $\FP$ is of Gibbs-type \eqref{gibssEPPF}, then
\[
\begin{split}
\P\{  \tilde \Pi_n=\pi_n\}   & = (1-a)^{k}  V_{n,k} \prod_{j=1}^k (1-\sigma)_{n_j-1} \\
& 
 +(1-a)^{k-1}
 \sum_{i=1}^k   \prod_{j=1, j \not=i}^k (1-\sigma)_{n_j-1}
 \sum_{r=1}^{n_i}  a^r   V_{n,k-1+r} S_{\sigma}(n_i,r).  \\
\end{split}
\]
\end{proposition}

\begin{proof}
In this case  
$H^\#(\bmm)=0$ if $m_i\geq 2$ and  $m_j \geq 2$ for some $i \not =j$ because $H$ has only one atom. 
Moreover, $H^\#(\bmm)$ is clearly symmetric and 
\[
H^\#(1,1,1,\dots,1)=(1-a)^k+k(1-a)^{k-1}a
\]
\[
H^\#(m,1,\dots,1)=a^m(1-a)^{k-1} \qquad \text{for $m >1$}.
\]
By Proposition \ref{prop_partition} 
\[
\begin{split}
\P\{  \tilde \Pi_n=\pi_n\}   & = [(1-a)^{k}+k(1-a)^{k-1}a]\FP(n_1,\dots,n_k) 
\\
& 
 +(1-a)^{k-1}\sum_{i=1}^k\sum_{m_i=2}^{n_i}a^{m_i}\sum_{  \blam  \in \Lambda(\bmm) } c(\blam) \tilde \FP (\blam) \\
 &=[(1-a)^{k}+k(1-a)^{k-1}]\FP(n_1,\dots,n_k)+\\
 &+(1-a)^{k-1}\sum_{i=1}^k\sum_{r=2}^{n_i}a^r\sum_{
  \substack{
 \blam\in \Lambda(\bmm)  \, \text{for $ \bmm$:} \\  m_i=r,\ m_j=1, j\neq i} 
 }  
 c(\blam)  \FP (n_1,\dots, n_{i-1}, {\mathbf n}_r^{(i)},n_{i+1},\dots, n_k) \\
 &=(1-a)^{k}\FP(n_1,\dots,n_k)\\
 &+(1-a)^{k-1}\sum_{i=1}^k\sum_{r=1}^{n_i}a^r\sum_{ \substack{
 \blam\in \Lambda(\bmm)  \, \text{for $ \bmm$:} \\  m_i=r,\ m_j=1, j\neq i} }  c(\blam)  \FP (n_1,\dots, n_{i-1}, {\mathbf n}_r^{(i)},n_{i+1},\dots, n_k) 
\end{split}
\]
where ${\mathbf n}_r^{(i)}$ is any vector of $r$ positive integers with sum $n_i$ and such that $\lambda_{ij}$ of them are equal to $j$. In view of the definition of $c(\blam)$, formula \eqref{tilde_p_dens} is immediately obtained. 

If  $\FP$ is of Gibbs-type, taking into account \eqref{stirling}, then 
\[
 q_n(r|n_1,\dots, n_{i-1},n_{i+1},\dots,n_k )=\frac{V_{n,k-1+r}}{V_{n-n_i,k-1}} S_{\sigma}(n_i,r)
\]
and the second part of the thesis follows by simple algebra. 
\end{proof}

Applying Proposition \ref{SpikeSlabEPPF} to the Pitman-Yor  EPPF    defined in \eqref{PYeppf}, 
one immediately  recovers the results stated in Theorem 1 and Corollary 1 of  \cite{Canale2017}.

\section{Asymptotics distribution of the number of clusters}

An exchangeable random partition $\Pi=(\Pi_n)_{n \geq 1}$ is said to have  asymptotic 
diversity $S$ if  
\begin{equation}\label{sigmadiversity}
   \frac{|\Pi_n|}{c_n }\to S \quad a.s. 
\end{equation}
for  a strictly  positive random variable $S$ and a suitable (deterministic) normalizing  sequence $(c_n)_{n\geq 1}$. This definition  
generalizes   
the concept of $\sigma$-diversity, which is  \eqref{sigmadiversity} for $c_n=n^{\sigma}$,
see  Definition 3.10 in  \cite{Pit06}. 
There are important examples of exchangeable random partitions in which 
$c_n=n^{\sigma} \ell(n)$ for some $\ell$ slowly varying  at infinity. 
In particular,  if    $(\Pi_n)_{n \geq 1}$ is an exchangeable random partition with EPPF of  Gibbs-type 
\eqref{gibssEPPF}, then  \eqref{sigmadiversity}
holds with 
 \[
c_n:=\left \{ 
 \begin{array}{ll}
1 & \text{if $\sigma<0$} \\
\log(n) & \text{if $\sigma=0$} \\
n^\sigma & \text{if $0<\sigma<1$}, \\
 \end{array}
 \right . 
 \]
 see Section 6.1 of \cite{Pitman2003}. 

In this Section we investigate the asymptotic diversity 
for a  random  partition $\tilde \Pi=(\tilde \Pi_n)_{n \geq 1}$  induced by a  $gSSS(\FP,H)$ 
for a general $H$.  

By Proposition \ref{prop_gsss}, we may assume that  $\tilde \Pi$ is the random partition induced by 
a sequence
\[
(\xi_n)_{n\geq 1}=(Z_{\SC_n(\Pi)})_{n \geq 1}
\]
where 
  $\Pi=(\Pi_n)_{n \geq 1}$ is a random partition with EPPF equal to $\FP$,
and $(Z_n)_{n \geq 1}$ is a sequence of i.i.d. random variables with distribution $H$, independent from 
$\Pi$.

Recalling \eqref{acca},  we also write 
 \begin{equation}
 H(dx)= a  H^{d}(dx)+ (1-a) H^{c}(dx)
\end{equation}
where, if $a>0$,  
\[
H^{d}(dx) =\sum_{i=1}^{+\infty} \frac{\bar a_i}{a} \delta_{\bar x_i}(dx).
\]
%

 For $a=0$ we recover the classical case of a diffuse base measure were $|\tilde \Pi_n|=|\Pi_n|$ a.s..  
Since this  case is well studied, from now on we assume $a>0$. 

Set
 \[
 K_n=| \Pi_n| \qquad \text{and}\qquad
N_n=\sum_{j=1}^{K_n} \J\{Z_j \in \XX_0^c\}.
\]
Hence  $N_n$ is 
 the
 random number of elements in $(Z_1,\dots,Z_{K_n})$ sampled from 
the diffuse component $H^c$ and  $K_n-N_n$ is  the number of elements sampled from the discrete component $H^{d}$.
 
Let $\delta_1,\delta_2,\dots$ be the indexes corresponding 
 to the $Z_j \in \XX_0$, i.e.
 \[
\delta_1=\inf\{i:Z_i \in \XX_0\}, \quad \delta_k=\inf\{i > \delta_{k-1}:Y_i\in \XX_0\} \quad k\geq 2. 
 \]
For any set of points $(x_1,\dots,x_n)$ in $\XX$, let 
 $\Lambda(x_1,\dots,x_n)$ be the number of different elements in $(x_1,\dots,x_n)$,
and define
 \[
\Lambda_n=  \Lambda(Z_{\delta_1},\dots,Z_{\delta_{K_n-N_n}}).
 \]
 One can check that 
 \begin{equation}\label{identity1}
|\tilde \Pi_n| 
=N_n+ \Lambda_n \quad \text{a.s.}.
\end{equation}

Note  that if $a=1$ then $N_n=0$  and, with probability one, 
$
|\tilde \Pi_n|  = \Lambda(Z_1,\dots,Z_{K_n})$.

It is easy to determine the asymptotic behavior of the first term in \eqref{identity1}, if \eqref{sigmadiversity} holds true.
Using the fact that $K_n \to +\infty$, 
the Strong Law of Large Numbers gives
\[
\lim_n \frac{1}{K_n} \sum_{j=1}^{K_n} \J\{Z_j \in \XX_0^c\}= \E[\J\{Z_1 \in \XX_0^c\}]=(1-a) \quad \text{a.s.}.\] Since
 \[
\frac{N_n}{c_n}= \frac{K_n}{c_n}  \frac{1}{K_n} \sum_{j=1}^{K_n} \J\{Z_j \in \XX_0^c\}
\]
by \eqref{sigmadiversity} one obtains, for $n\to +\infty$, 
 \begin{equation}\label{convND}
 \frac{N_n}{c_n}\to (1-a )S \quad \text{a.s.}
\end{equation}

%

This allows to easily obtain a first convergence result in the case
 $\XX_0=\{\bar x_1,\dots,\bar x_M\}$ is a finite set, as happens for spike and slab base measures 
described in Subsection \ref{S:spikeslab}.

\begin{proposition}\label{PropAs1}
 Assume that 
\eqref{sigmadiversity} holds true with $c_n \to +\infty$
and $|\XX_0|<+\infty$, then 
\begin{equation}\label{convtildeKn} 
\frac{|\tilde \Pi_n|}{c_n} \to (1-a)S \quad \text{a.s.  if  $a<1$}
\end{equation}
and 
\begin{equation}\label{convtildeKnfinito2} 
{|\tilde \Pi_n|} \to |\XX_0| \quad \text{a.s.  if  $a=1$}.
\end{equation}
\end{proposition}

\begin{proof}
 From  \eqref{convND} one deduces that $K_n-N_n \to +\infty$, hence 
${\Lambda_{n}} \to |\XX_0|<+\infty$ a.s.. If $a=1$, 
\eqref{convtildeKnfinito2} follows since in this case $|\tilde \Pi_n|=\Lambda_{n}$.
If
$0<a<1$,  \eqref{convtildeKn}
follows from \eqref{convND} and  \eqref{identity1} since 
${\Lambda_{n}}/{c_n} \to 0$ a.s.. 
\end{proof}

The next results describe the situation in which $H^d$ is supported by an infinite set. In this case the asymptotic behaviour of
$|\tilde \Pi_n|$ is related to the behaviour of the number of different elements in an i.i.d. sample from $H^d$. 
Define
\[
Z_j^*= \left \{
\begin{array}{cc}
Z_j  & \text{if $Z_j \in \XX_0$} \\
\bar x_1 & \text{if $Z_j \not \in \XX_0$}  \\  
\end{array}
\right . 
\]
and set 
\[
L_n :=\Lambda(Z_1^*,\dots,Z_n^*). 
\]

In what follows we need the following  assumption:
%
%
%
\begin{equation*}(\text{\bf H})\label{hp}
\begin{split}
&\text{$L_n/b_n \to z_0>0$ a.s. for 
$b_n=n^{\sigma_0} \ell_0(n)$,} \\
& \text{where
$\sigma_0 \in [0,1]$ and  $\ell_0$ is 
slowly varying at $+\infty$ such that}\\ 
&\text{$\ell_0(n) \to 0$ as $n \to +\infty$  if $\sigma_0=1$,
$\ell_0(n) \to +\infty$ as $n \to +\infty$ if $\sigma_0=0$. }
\end{split}
\end{equation*}  

\begin{lemma}\label{proptildeK_n} 
Assume 
\eqref{sigmadiversity} with $c_n \to +\infty$ and {\rm({\bf H})}.
Then 
\begin{equation}\label{convtildeKn2}
\begin{split}
&\frac{|\tilde \Pi_n|}{c_n} \to (1-a)S\quad \text{a.s. if }\quad 0<a<1;\\ 
&\frac{|\tilde \Pi_n|}{b_{c_n}} \to  z_0 S^{\sigma_0}  \quad \text{a.s. if } \quad a=1.
\end{split}
\end{equation}
\end{lemma}

\begin{proof} 
Write
\[
\frac{L_{K_n}}{b_{c_n}}= \frac{L_{K_n}}{b_{K_n}}\frac{b_{K_n}}{b_{c_n}} 
 = \frac{L_{K_n}}{b_{K_n}}  \left( \frac{{K_n}}{{c_n}} \right )^{\sigma_0} 
\frac{\ell_0 \left(\frac{{K_n}}{c_n} c_n  \right)}{\ell_0({c_n})}. 
\]
Since,   $K_n/c_n \to S>0 $ a.s. and  $K_n \to +\infty$, 
 using the fact that $L_n/b_n \to z_0$ a.s. one gets
that 
\[
\frac{L_{K_n}}{b_{K_n}}  \to z_0 \quad \text{and} \quad   \left( \frac{{K_n}}{{c_n}} \right )^{\sigma_0}  
\to S^{\sigma_0}\quad \text{a.s..}
\]
Recalling that for any slowly varying function 
$\ell_0(x_n y_n)/\ell_0(y_n) \to 1$ whenever $y_n \to +\infty$ and $x_n \to x>0$ (see
Theorem B.1.4 in \cite{deHaanFerreira}), one obtains
\[
\frac{\ell_0 \left(\frac{{K_n}}{c_n} c_n  \right)}{\ell_0({c_n})} \to 1 \quad a.s. 
\]
In conclusion, 
\begin{equation}\label{asymLKn}
\frac{L_{K_n}}{b_{c_n}} \to  z_0 S^{\sigma_0}  \quad a.s. 
\end{equation}
Since $L_{K_n}-1 \leq \Lambda_n \leq L_{K_n}$,
and $b_{c_n} \to +\infty$,  \eqref{asymLKn}
yields
\begin{equation}\label{convL}
\frac{\Lambda_{n}}{{b_{c_n}}} \to  z_0 S^{\sigma_0}  \quad \text{a.s.}
\end{equation}
Assume $a<1$ and write  
\[
\frac{|\tilde \Pi_n|}{c_n} = \frac{N_n}{c_n} + \frac{\Lambda_{n}}{b_{c_n}} \frac{b_{c_n}}{c_n}.
\]
We know that  $N_n/c_n \to (1-a)S$ a.s. (see \eqref{convND}), so that  combining \eqref{convL}
with 
\[
\frac{b_{c_n}}{c_n}= c_n^{\sigma_0-1} \ell_0(c_n) \to 0
\] 
we get the thesis. 
If $a=1$, $|\tilde \Pi_n|=\Lambda_n$ and the thesis is \eqref{convL}. 
\end{proof}

The main result  on $|\tilde \Pi_n|$  is Theorem \ref{PropAs2} below, 
obtained combining Lemma \ref{proptildeK_n} and well-known results on the 
number of  occupied cells in urn schemes obtained by  \cite{Karlin67}
and reviewed in Proposition \ref{lemmakarlin} in the Appendix. 
For every $x>0$, define
\begin{equation}\label{defalpha}
\alpha(x) :=\#\{ j : a_j \geq 1/x\}.
\end{equation}

 Combining Lemma \ref{proptildeK_n} and Proposition \ref{lemmakarlin}, one easily obtain the following result.

\begin{thm}\label{PropAs2}
Assume that 
\eqref{sigmadiversity} holds true with $c_n \to +\infty$
and that $\alpha(x)=x^{\sigma_0} \ell_0^*(x)$
where $0 \leq \sigma_0 \leq 1$  and  $\ell_0^*$ is a slowly varying function at $+\infty$ {\rm(}with
$\lim_{x \to +\infty}  \ell_0^*(x)=+\infty$ if $\sigma_0=0${\rm)}. 
Then, 
\begin{equation}\label{convtildeKn3} 
\frac{|\tilde \Pi_n|}{c_n} \to (1-a)S \quad \text{a.s.  if  $a<1$}
\end{equation}
and
 \begin{equation}\label{convtildeKn4} 
\frac{|\tilde \Pi_n|}{{c_n^{\sigma_0} \ell_0(c_n) } } \to  z_0 S^{\sigma_0}  \quad \text{a.s. if $a=1$}
\end{equation}
where
\begin{itemize}
\item 
$\ell_0(x)=\ell_0^*(x)$ and $z_0= \Gamma(1-\sigma_0)$ if $0 \leq \sigma_0 <1$;
\item  $\ell_0(x)=\int_{x}^{+\infty} u^{-1} \ell_0^*(u)du<+\infty$  and $z_0=1$ if  $\sigma_0=1$. 
\end{itemize}
\end{thm}

The last result of this Section (see Theorem \ref{PropAs3}) concerns the asymptotic behaviour  of 
 the number of blocks  with $r$ elements in $\tilde \Pi_n$, i.e. 
\[
\mathcal{K}_{r}(\tilde \Pi_n):= \#\{ j=1,\dots,|\tilde \Pi_n|: |\tilde \Pi_{j,n}|=r  \},
\]
for any $r=1,2,\dots,n$. We start with

\begin{lemma}\label{lemmaKnr}
Assume that \eqref{sigmadiversity} holds true with $c_n=n^{\sigma} \ell(n)$
where $0 < \sigma < 1$  and  $\ell$ is a deterministic slowly varying function at $+\infty$. 
Assume also that {\rm({\bf H})} holds. 
Then, if $a<1$, for every $r\geq 1$, 
\[
\frac{\mathcal{K}_{r}(\tilde \Pi_n) }{c_n} \to (1-a)  \frac{\sigma\Gamma(r-\sigma)}{\Gamma(1-\sigma) r!}S  \quad \text{a.s.}
\] 
\end{lemma}

\begin{proof} Note that
\begin{equation}\label{knr}
\mathcal{K}_{r}(\tilde \Pi_n)= \sum_{j=1}^{K_n} \J\{Z_j \in \XX_0^c\}\J\{ |\Pi_{j,n}|=r  \} + \Delta_{n,r}, 
\end{equation}
where $\Delta_{n,r}$ is the number of blocks with $r$ elements in $\tilde \Pi_n$ which are derived 
by merging two or more blocks of $\Pi_n$.
In the proof of Lemma \ref{proptildeK_n}, we have already shown that 
$\Lambda_n/c_n \to 0$ a.s.. Since 
$\Delta_{n,r} \leq \Lambda_n$ one gets 
$\Delta_{n,r}/c_n \to 0$ a.s.. 
Now, by \eqref{knr}, one has 
\[
\frac{\mathcal{K}_{r}(\tilde \Pi_n)}{c_n}=(1-a) \frac{K_{n,r}}{c_n} + S_n \frac{K_n}{c_n} +\frac{\Delta_{n,r}}{c_n}
\]
where 
\[
\begin{split}
K_{n,r}& =\mathcal{K}_{r}(\Pi_n) \\
S_n & = \sum_{j=1}^{K_n} U_j a_{n,j}  \\
U_j & =\J\{Z_j \in \XX_0^c\}-(1-a) \\
a_{n,j}&=\frac{\J\{ |\Pi_{n,j}|=r  \} }{K_n}. \\
\end{split}
\]
Recalling that $K_n/c_n \to S$ a.s.,
the thesis follows from  Proposition \ref{lemma3}, if we prove that 
$S_n \to 0$ a.s.

Since $ S>0$ a.s. and $c_n=n^\sigma\ell(n)$, 
 there is a finite random variable $T$ and $0\leq \eps<\sigma$ such that
\[
\frac{1}{K_n}   \leq \frac{T}{n^{\sigma-\eps}} 
\]
for every $n$ with probability one.
Hence
\begin{equation}\label{boundAn1}
\sum_{j \geq 1} a_{n,j}^2 = \frac{K_{n,r}}{K_n^2} \leq \frac{1}{K_n}\leq \frac{T}{n^{\sigma-\eps}} 
\end{equation}
and 
\begin{equation}\label{boundAn2}
a_{n,j} \leq \frac{1}{K_n}\leq \frac{T}{n^{\sigma-\eps}}.
\end{equation}
The thesis now follows by \eqref{boundAn1}-\eqref{boundAn2} in combination with 
 Corollary 2 in \cite{Stout68}, reported as Proposition \ref{Prop:Stout} in Appendix.
To be more explicit,  let  $\CG$ the $\sigma$-field generated by
$(\Pi_n)_{n\geq 1}$ and set 
$U=(U_j)_{j\geq 1}$. Write $S_n=F_n(U,A)$ where  $A:=[K_n,a_{n,1},\dots,a_{n,K_n}]_{n \geq 1}$ 
and  $F_n$ is a  deterministic function. 
Since
$A$  is  $\CG$-measurable and $U$ is independent from $\CG$,  a regular version of the conditional law of $U$ given $\CG$ is $\P_{U|\CG}(du|\omega)=Q(du)$
where  $Q(du)$ is the law of $U$. 
At this stage note that 
\[
\begin{split}
\P\{ \lim_n S_n=0 |\CG \}(\omega) & =\int \J\{ \lim_n F_n(u,A(\omega))=0\} \P_{U|\CG}(du|\omega) \\
& 
=\int \J\{ \lim_n F_n(u,A(\omega))=0\}  Q(du)=Q\{u: \lim_n F_n(u,A(\omega))=0\}.
\\
\end{split} 
 \]
By \eqref{boundAn1}-\eqref{boundAn2},  for $\omega$ in  a set of $\P$-probability one, 
$\tilde a_{n,j}:=a_{n,j}(\omega)$ satisfies \eqref{stoutSLLN0}
and hence, by \eqref{stoutSLLN}, 
$Q\{ \lim_n F_n(u,A(\omega))=0\}=1$, which shows that 
$\P\{ \lim_n S_n=0 \}=1$. 
\end{proof}

{Combining Lemma \ref{lemmaKnr} and Proposition \ref{lemmakarlin}  in Appendix we can easily state  next result.}

\begin{thm}\label{PropAs3}
Assume that \eqref{sigmadiversity} holds true with $c_n=n^{\sigma} \ell(n)$
where $0 < \sigma < 1$  and  $\ell$ is a deterministic slowly varying function at $+\infty$,
and that $\alpha(x)=x^{\sigma_0} \ell_0^*(x)$
where $0 \leq \sigma_0 \leq 1$  and  $\ell_0^*$ is a slowly varying function at $+\infty$ (with
$\lim_{x \to +\infty}  \ell_0^*(x)=+\infty$ if $\sigma_0=0$). 
If $a<1$, then
\[
\frac{\mathcal{K}_{r}(\tilde \Pi_n) }{n^\sigma \ell(n)} \to (1-a) S \frac{\sigma\Gamma(r-\sigma)}{\Gamma(1-\sigma) r!} \quad \text{a.s.}
\] 
 for every $r\geq 1$. 
\end{thm}

\appendix
\section{}

\subsection{Exchangeable  random partitions} 

In this section we collect some definitions and  well-known results concerning 
exchangeable  random partitions. We refer to  Chapter 11 \cite{Aldous85} and 
\cite{Pit06} for the proofs and further details. 

Let $\nabla:=\{ p_j^{\downarrow} \in [0,1]: p_1^{\downarrow} \geq p_2^{\downarrow} \geq \dots,  \sum_{j \geq 1} p_j^{\downarrow} \leq 1\} $.
We start by recalling  Kingman's theorem. 

\begin{proposition}[\cite{Kingman78}]\label{kingmanth}
Given any exchangeable random partition $\Pi$
 with EPPF $\FP$, denote 
 by $\Pi_{j,n}^{\downarrow}$ {the blocks of the partition rearranged in decreasing order with respect to number of element in the blocks of $\Pi_n$}. Then 
\begin{equation}\label{kingmanTh}
\lim_n \Big (\frac{|\Pi_{j,n}^{\downarrow}|}{n} \Big)_{j \geq 1}= (p_j^{\downarrow})_{j \geq 1} \quad a.s. 
\end{equation}
for some random $p^{\downarrow}= (p_j^{\downarrow})_{j \geq 1}$ taking values in $\nabla$.
Moreover, $\mathfrak{K}(\FP):=\text{Law}(p^{\downarrow})$   defines  a bijection 
from the set of the EPPF and the laws on $\nabla$.
\end{proposition} 

As a consequence, one obtains the following results. 

\begin{itemize}
\item[(A1)] Let $p=(p_j)_{j \geq 1}$ be a sequence of random weights in $[0,1]$ such that $\sum p_j =1$ a.s..
Denote by  $p^{\downarrow}$  the sequence obtained by rearranging $p$  in decreasing order and set 
$\FP=\mathfrak{K}^{-1}(\text{Law}(p^{\downarrow}))$.
Then the random partition induced by a  sequence $(I_n)_{n\geq 1}$ which is 
conditionally i.i.d. given $p$ with conditional distribution 
\[
P\{ I_n=j| p \} = p_j  \quad a.s. 
\]
has EPPF $\FP$. 
\item[(A2)] If  
$\Pi$ is a random partition with EPPF $\EPk$ and the sequence 
$p^{\downarrow}$ defined in \eqref{kingmanTh}
 satisfies $\displaystyle \sum_jp_j^{\downarrow} =1 \; a.s.$, then one can define   a sequence 
 $(I_n)_n$ of integer-valued random variables, conditionally i.i.d.  given $p^{\downarrow}$, with $P\{ I_n=j| p^{\downarrow} \} = p^{\downarrow}_j$, such that $\Pi(I)=\Pi,\ a.s.$. 
\end{itemize}

\subsection{Further useful results.}

In order to verify assumption {\rm({\bf H}) one 
can use well-known results from   \cite{Karlin67}. 

\begin{proposition}\label{lemmakarlin}
Let $\alpha$ be defined in \eqref{defalpha} and 
assume that 
 $\alpha(x)=x^{\sigma_0} \ell_0^*(x)$
where $0 \leq \sigma_0 \leq 1$  and  $\ell_0^*$ a slowly varying function at $+\infty$.
Then,  {\rm({\bf H})} holds true 
with 
\begin{itemize}
\item 
$\ell_0(x)=\ell_0^*(ax)$ and $z_0=a^{\sigma_0} \Gamma(1-\sigma_0)$ if $0 \leq \sigma_0 <1$;
\item  $\ell_0(x)=\int_{ax}^{+\infty} u^{-1} \ell_0^*(u)du<+\infty$  and $z_0=a$ if  $\sigma_0=1$. 
\end{itemize}
\end{proposition}

\begin{proof}  
Now note that $\P\{Z_n^*=j\}=a a_j \J\{j \not =1\} + (1-a \sum_{k \geq 2} a_k)\J \{ j=1\}=:a_j^*$. Hence, 
for $x$ big enough 
\begin{equation}\label{alphastar}
\alpha^*(x) :=\#\{ j : a_j^* \geq 1/x\}=\alpha(ax).
\end{equation}
The thesis follows now from Theorem 8 and Theorem $1^\prime$ in  \cite{Karlin67}. 
The expression for $\ell_0$ in the case  $\sigma_0=1$ is slightly different 
from the one used in Theorem $1^\prime$ of \cite{Karlin67} and it is taken from 
Proposition 14 in \cite{Gnedin07}.
\end{proof}

\begin{proposition}[\cite{Gnedin07}]\label{lemma3}
Assume that \eqref{sigmadiversity} holds true with $c_n=n^{\sigma} \ell(n)$
where $0 < \sigma < 1$  and  $\ell$ is a deterministic slowly varying function at $+\infty$. 
Then, for any $K>1$, 
\[
\frac{1}{n^\sigma \ell(n)}(|\Pi_n|,\mathcal{K}_{n,1}(\Pi_n),\dots,\mathcal{K}_{n,K}(\Pi_n)) 
\to S  \Big (1, \frac{\sigma \Gamma(1-\sigma)}{\Gamma(1-\sigma) 1!}, \dots, \frac{\sigma \Gamma(K-\sigma)}{\Gamma(1-\sigma) K!} \Big ) \quad \text{a.s.}
\] 
\end{proposition}

\begin{proof}
When $\ell$ is a constant then the thesis follows from the results of Section 
10 in \cite{Gnedin07}, see in particular formula  (51). Minor modifications 
of the arguments used in Section 
10 in \cite{Gnedin07} yields the results for a general $\ell$. 
\end{proof}

For ease of reference, we report below 
Corollary 2 of \cite{Stout68}.

\begin{proposition}[\cite{Stout68}]\label{Prop:Stout} If $(D_j)_{j \geq 1}$ are i.i.d. bounded random variables with zero mean 
 and $(\tilde a_{n,j})_{n \geq 1, j \geq 1}$ are deterministic weights such that
  \begin{equation}\label{stoutSLLN0}
\text{ $\sum_{j \geq 1} \tilde a_{n,j}^2 \leq C/n^{\alpha}$ and $\tilde a_{n,j} \leq C/n^\alpha$ with $0<\alpha<1$,}
 \end{equation}
 then for $n \to +\infty$
 \begin{equation}\label{stoutSLLN}
 \sum_{j \geq 1} \tilde a_{n,j} D_j  \to 0 \quad {a.s.}.
 \end{equation}
 \end{proposition}

\bibliographystyle{apalike}

\end{document}